\documentclass[11pt]{amsart}
\usepackage{geometry}
\usepackage{amssymb}

\DeclareMathOperator{\ab}{Ab}
\DeclareMathOperator{\uni}{uni}
\DeclareMathOperator{\multi}{multi}
\DeclareMathOperator{\Aut}{Aut}
\DeclareMathOperator{\rk}{rk}

\usepackage{amsrefs}

\newtheorem{theorem}{Theorem}[section]
\newtheorem{lemma}[theorem]{Lemma}
\newtheorem{corollary}[theorem]{Corollary}
\newtheorem{proposition}[theorem]{Proposition}

\theoremstyle{definition}
\newtheorem{definition}[theorem]{Definition}
\newtheorem{example}[theorem]{Example}

\newtheorem*{ack}{Acknowledgements}

\theoremstyle{remark}
\newtheorem{remark}[theorem]{Remark}

\numberwithin{equation}{section}

\begin{document}

\title{Vanishing of Higher Order Alexander-type Invariants of Plane Curves}


\author{Jos\' e I. Cogolludo-Agust\' in}
\address{Departamento de Matem\'aticas, IUMA, Universidad de Zaragoza, 
C. Pedro Cerbuna 12, 50009 Zaragoza, Spain}
\email{jicogo@unizar.es}

\author{Eva Elduque}
\address{Department of Mathematics,
University of Michigan-Ann Arbor,
530 Church Street, 
Ann Arbor, MI 48109-1043, USA}
\email{elduque@umich.edu}
\thanks{The first author is partially supported by
Grupo ``\'Algebra y Geometr{\'\i}a'', Gobierno de Arag\'on/Fondo Social
Europeo and by the Spanish Government MTM2016-76868-C2-2-P. The second author is partially supported by an AMS-Simons Travel Grant}


\date{}

\dedicatory{}

\commby{}

\begin{abstract}
The higher order degrees are Alexander-type invariants of complements to an affine plane curve.
In this paper we characterize the vanishing of such invariants for transversal unions of plane curves 
$C'$ and $C''$ in terms of the finiteness, the vanishing properties of the invariants of $C'$ and $C''$, 
and whether they are irreducible or not. As a consequence, we characterize which of these types of 
curves have trivial multivariable Alexander polynomial in terms of their defining equations. 
Our results impose obstructions on the class of groups that can be realized as fundamental groups of 
complements of a transversal union of curves.
\end{abstract}

\maketitle

\section{Introduction}
The study of curve complements goes back to the work of Zariski (\cite{zariski}), who observed that the position of the singularities of a plane curve influenced the topology of the curve, and that the fundamental group of the complement of the curve detected this phenomenon.  Alexander-type invariants, which appeared first in classical knot theory and were first imported to study singularities of plane curves by Libgober (\cites{lib3,libalex}), are easier to handle than the fundamental group, and they are also sensitive to the type and position of singularities.


In knot theory, a strategy to address problems that the Alexander polynomial is not strong enough to solve is to consider non-abelian Alexander-type invariants, such as the \textit{higher order degrees} (e.g., see Cochran, \cite{Cochran}), which have been shown to give better bounds for the knot genus than the Alexander polynomial (Horn, \cite{Horn}). These invariants also have striking applications in the world of $3$-manifolds (Harvey, \cite{Harvey}).

Leidy and Maxim initiated in \cite{MaxLeidy, MaxLeidySurvey} the study of higher order Alexander-type invariants for complex affine plane curve complements, and Maxim and the second author continued this work in \cite{EvaMax}. To any affine plane curve $C \subset \mathbb{C}^2$ (given by the zeros of a reduced non-constant polynomial $f\in\mathbb{C}[x,y]$), in \cite{MaxLeidy} one associates a sequence $\{\delta_n(C)\}_n$ of (possibly infinite) integers, called the higher order degrees of $C$. Roughly speaking, these integers measure the
``sizes'' of quotients of successive terms in the rational derived series $\{G_r ^{(n)}\}_{n \geq 0}$ of the fundamental group $G=\pi_1(\mathbb{C}^2~\setminus~C)$ of the curve complement. 
It was also noted in
\cite{MaxLeidy} that the higher order degrees of plane curves (at any level
$n$) are sensitive to the position of singular points.
These integers can also be interpreted as L2-Betti numbers associated to the tower of coverings of $\mathbb{C}^2 \setminus C$ corresponding to the subgroups $G_r ^{(i)}$ (the first of which is the universal abelian cover), so in principle there is no reason to expect that such invariants have any good vanishing or finiteness properties. 
Some finiteness results obtained in \cites{MaxLeidy,EvaMax} are summarized in this theorem.

\begin{theorem}\label{thmfinite}
Let $C\subset \mathbb{C}^2$ be a plane curve of degree $m$. If one of the following conditions hold, then $\delta_n(C)$ is finite:
\begin{enumerate}
\item $C$ is irreducible \cite[Remark 3.3]{MaxLeidy}.
\item \label{trans}$C$ is in general position at infinity \cite[Corollary 4.8]{MaxLeidy}.
\item \label{ess}$C$ is an essential line arrangement \cite[Theorem 2]{EvaMax}.
\item \label{control}$C$ has only nodes or simple tangents at infinity  \cite[Theorem 4]{EvaMax}.
\end{enumerate}

Moreover, in the cases~\eqref{trans}, \eqref{ess}, and~\eqref{control}, we have that
$
\delta_n(C)\leq m(m-2)
$
for all $n\geq 0$. That is, there is a uniform bound for all the higher order degrees that depends only 
on the degree of a curve.
\end{theorem}

In relation to an old question of Serre (\cites{arapura,Serre}), finiteness results impose restrictions 
on which groups can be realized as fundamental groups of curve complements, but vanishing results, on top 
of being stronger, shed more light on what type of problems these invariants are well suited for. For example, 
if we know that $\delta_n=0$ for a class of curves, then $\delta_n$ will not distinguish curves within that 
class, but it can potentially distinguish a curve in that class from a curve in a different class.

An example  of a ``vanishing'' (or ``triviality'') result is the following theorem of Oka. In general terms, 
it tells us that the univariable Alexander polynomial (see Definition~\ref{defnuni}) of a transversal union 
of curves $C=C'\cup C''$ does not remember information about the topology of $C'$ or $C''$, even though the 
fundamental group does (See Theorem~\ref{thmos}).

\begin{theorem}[\cite{Oka}, Theorem 34]
Let $C$ be a plane curve of the form $C=C'\cup C''$, where $C'$ and $C''$ are curves in $\mathbb{C}^2$ of degrees $m'$ and $m''$ respectively. Assume that $C$ is in general position at infinity, and assume that $C'\cap C''$ consists on $m'm''$ distinct points. Then,
$$
\Delta^{\uni}_C(t)=(t-1)^{s-1}
$$
where $s$ is the number of irreducible components of $C$ and $\Delta^{\uni}_C(t)$ is the (univariable) 
Alexander polynomial of $C$.
\end{theorem}

It is natural to ask whether more involved Alexander invariants also exhibit this behavior. In this paper, we completely characterize the vanishing of the higher order degrees of a union $C$ of two curves $C'$ and $C''$ that intersect transversally (even when $C$ is not in general position at infinity). This characterization is done in terms of the finiteness and vanishing properties of the higher order degrees of $C'$ and $C''$, obtaining vanishing results in most cases (and finiteness in all cases). More concretely, we obtain the following.

\begin{theorem}\label{thmCharac}
Let $C=C'\cup C''\subset \mathbb{C}^2$ be the union of two affine plane curves, with $\deg C'=m'$ and $\deg C''=m''$. Suppose that $C'\cap C''$ consists on $m'm''$ distinct points in $\mathbb{C}^2$. Then,
\begin{enumerate}
\item $\delta_n(C)$ is finite for all $n\geq 0$.
\item If $C'$ and $C''$ are both irreducible or both not irreducible, then $\delta_n(C)=0$ for all $n\geq 0$.
\item If $C'$ is irreducible and $C''$ is not irreducible, and
\begin{enumerate}
\item $\delta_0(C')\neq 0$, then $$ \delta_0(C)=0\Leftrightarrow \delta_0(C'')<\infty, 
\text{ \ \ and \ \ \ }\delta_n(C)=0 \text{ \ \ for all \ \ }n\geq 1.$$
\item $\delta_0(C')= 0$, then, for all $n\geq 0$, $$\delta_n(C)=0\Leftrightarrow \delta_n(C'')<\infty.$$
\end{enumerate}
\end{enumerate}
\end{theorem}
This provides a broad generalization of the vanishing results of \cite{EvaMax}, where the fundamental 
group of one of the curve complements was assumed to be isomorphic to $\mathbb{Z}$, and $\delta_n$ of 
the other curve was assumed to be finite.

The paper is structured as follows. In Section 2 we recall the relevant definitions of the Alexander-type invariants that are used throughout the paper and the relationships between them. In Section 3 we prove the main result (Theorem~\ref{thmCharac}). In Section 4, we characterize which curves have $\delta_0(C)=\infty$ in terms of their defining equations (curves of affine pencil type, as defined in Lemma~\ref{lempencil}) and arrive at Corollary~\ref{corAP} below about the triviality of the multivariable Alexander polynomial of a transversal union of curves (see Definition~\ref{defnmulti}). This corollary provides concrete restrictions as to which groups can be realized as fundamental groups of a complement of a transversal union of curves (see Remark~\ref{remcomputations}).
\begin{corollary}\label{corAP}
Under the same hypotheses as in Theorem~\ref{thmCharac},
\begin{enumerate}
\item $\Delta_C^{\multi}\neq 0$.
\item If $C'$ and $C''$ are both irreducible or both not irreducible, then $\Delta_C^{\multi}$ is a non-zero constant.
\item If $C'$ is irreducible and $C''$ is not irreducible, then $\Delta_C^{\multi}$ is a non-zero constant if and only if $C''$ is not of affine pencil type.
\end{enumerate}
Moreover, if $\Delta_C^{\multi}(t,t_1,\ldots,t_s)$ is not a non-zero constant, then it is of the form $(t-1)^k$ for some $1\leq k\leq m''-1$, where $t$ is the variable corresponding to a positively oriented meridian around $C'$, and $t_i$ is the variable corresponding to a positively oriented meridian around the $i$-th irreducible component of $C''$.
\end{corollary}

\begin{ack}
The authors would like to thank Mois\'es Herrad\'on Cueto and Lauren\c{t}iu Maxim for useful discussions.
\end{ack}

\section{Definitions of classical and higher order Alexander invariants}

In this section we recall the basic definitions of the notions that will be used throughout this note. For a more detailed explanation of the different Alexander invariants used in this paper, we refer the reader to \cite{libalex} (for univariable Alexander polynomials), \cite{suciu} (for multivariable Alexander polynomials), and \cite{MaxLeidy} (for higher order degrees of plane curve complements), for example.

\subsection{Alexander Polynomials}

Let $C=\{f(x,y)=0\}\subset \mathbb{C}^2$ be a plane curve given by the zeros of a reduced polynomial $f$, with complement $U:=\mathbb{C}^2\setminus C$, and denote by 
$G:=\pi_1(U)$ the fundamental group of its complement.
If $C$ has $s$ irreducible components, then \begin{equation}\label{eq1} H_1(G;\mathbb{Z})=H_1(U;\mathbb{Z})=G/G'=\mathbb{Z}^s,\end{equation} 
is generated by meridian loops about the smooth parts of the irreducible components of~$C$.

Let $\psi$ be the linking number homomorphism $G \overset{\psi}{\rightarrow} \mathbb{Z}$, given by 
$\alpha \mapsto \text{lk}(\alpha,C).$
Since $f$ is a reduced polynomial, $\psi$ is the map induced in fundamental groups by $f:U\rightarrow \mathbb{C}^*$. Let $\ab:G\rightarrow \mathbb{Z}^s$ be the abelianization homomorphism, which sends a positively oriented meridian about the $i$-th component of $C$ to the $i$-th element of the canonical basis of $\mathbb{Z}^s$. Let $\mathcal{L}^{\psi}$ and $\mathcal{L}^{\ab}$ be the local systems of $\mathbb{Q}[t^{\pm1}]$-modules and $\mathbb{Z}[t_1^{\pm 1},\ldots,t_s^{\pm 1}]$-modules induced by $\psi$ and $\ab$ respectively. More explicitely, $\mathcal{L}^{\psi}$ and $\mathcal{L}^{\ab}$ are given by
$$
\begin{array}{ccc}
G & \rightarrow & \Aut(\mathbb{Q}[t^{\pm 1}])\\
\gamma & \mapsto &(1\mapsto t^{\psi(\gamma)})
\end{array}
$$
and
$$
\begin{array}{ccc}
G & \rightarrow & \Aut(\mathbb{Z}[t_1^{\pm 1},\ldots,t_s^{\pm 1}])\\
\gamma & \mapsto &(1\mapsto t^{\ab(\gamma)})
\end{array}
$$
where $t^{(a_1,\ldots,a_s)}:=t_1^{a_1}\cdot\ldots\cdot t_s^{a_s}$ for all $(a_1,\ldots,a_s)\in\mathbb{Z}^s$.

For the following two definitions, let $F_i(M)$ be the $i$-th Fitting ideal of a module $M$ over a commutative ring. 

\begin{definition}\label{defnuni}
The \textit{univariable Alexander polynomial} of $U$, denoted by $\Delta^{\uni}_C(t)$ is defined as
$$
\Delta^{\uni}_C(t):=\text{a generator of }F_0\left(H_1(U;\mathcal{L}^{\psi})\right)\in\mathbb{Q}[t^{\pm1}],
$$
which is well defined up to multiplication by a unit of $\mathbb{Q}[t^{\pm1}]$.
\end{definition}
\begin{definition}
The \textit{multivariable Alexander polynomial} of $U$, denoted by $\Delta^{\multi}_C(t)$ is defined as
$$
\Delta^{\multi}_C(t_1,\ldots,t_s):=\gcd\left(F_1\left(H_1(U, u_0;\mathcal{L}^{\ab})\right)\right)\in\mathbb{Z}[t_1^{\pm1},\ldots,t_s^{\pm1}],
$$
where $u_0$ is a base point. It is well defined up to multiplication by a unit of $\mathbb{Z}[t_1^{\pm1},\ldots,t_s^{\pm 1}]$.
\label{defnmulti}
\end{definition}

\begin{remark}
Note that $H_1(U;\mathcal{L}^\psi)\cong H_1(U^\psi;\mathbb{Q})$ (\cite[Theorem 2.1]{kl}) as modules over $\mathbb{Q}[t^{\pm 1}]$, where $U^\psi$ is the infinite cyclic cover of $U$ induced by $\ker \psi$, whose deck group is isomorphic to $\mathbb{Z}$. Also note that the definition and computations are easier in the univariable case because $\mathbb{Q}[t^{\pm1}]$ is a PID. In the definition of the higher order degrees, a (noncommutative) PID is constructed to help generalize this construction of the univariable Alexander polynomial to other covers of $U$ that lie above $U^\psi$.
\end{remark}

\begin{remark}
If $C$ is irreducible, both definitions coincide. Indeed, from \cite[Remark 10]{EvaMax}  we know that $\Delta^{\multi}_C(t)$ divides $\Delta^{\uni}_C(t)$ in $\mathbb{Q}[t^{\pm 1}]$, and the same argument of the proof of \cite[Theorem 11]{EvaMax} (for $m=1$) shows that both polynomials are the same up to multiplication by a unit in $\mathbb{Q}[t^{\pm1}]$.
\end{remark}
\subsection{Higher Order Degrees}

\begin{definition}\label{hodg} The {\it rational derived series} of the group $G$ is defined inductively by:
$G_r ^{(0)}=G$, and for $n \geq 1$,
$$G_r ^{(n)}=\{g \in G_r ^{(n-1)} \mid g^k \in
[G_r ^{(n-1)},G_r ^{(n-1)}], \ \text{for some} \ k \in \mathbb{Z} \setminus \{0\}
\}.$$ 
\end{definition}

It is easy to see that $G_r ^{(i)} \triangleleft G_r ^{(j)}
\triangleleft G$, if $i \geq j \geq 0$. 
The successive quotients of the rational derived series are
torsion-free abelian groups. In fact (cf. \cite[Lemma 3.5]{Harvey}),
$$G_r^{(n)}/G_r^{(n+1)} \cong \left(G_r^{(n)}/[G_r^{(n)},
G_r^{(n)}] \right)/\{\mathbb{Z}-\text{torsion}\}.$$ Therefore, for 
$G=\pi_1(\mathbb{C}^2 \setminus C)$ we get from~\eqref{eq1} that $G'=G_r^{(1)}$.

The use of the rational derived series instead of the usual derived series is needed in order to avoid
zero-divisors in the group ring $\mathbb{Z}\Gamma_n$, where $$\Gamma_n:=G/G_r ^{(n+1)}.$$ 
$\Gamma_n$ is a poly-torsion-free-abelian group (PTFA), that is, it
admits a normal series of subgroups such that each of the successive
quotients of the series is torsion-free abelian (\cite[Corollary 3.6]{Harvey}). Thus, $\mathbb{Z}\Gamma_n$
is a right and left Ore domain, so it embeds in its classical right
ring of quotients $\mathcal{K}_n$, which is a skew-field. Every module over $\mathcal{K}_n$ is a free module, and such modules have a well-defined rank $\text{rk}_{\mathcal{K}_n}$ which is additive on short exact sequences.



In \cite{MaxLeidy}, one associates to any plane curve $C$ a
sequence of non-negative integers $\delta_n(C)$ as follows. Since $G'$ is in the kernel of $\psi$ (the linking number homomorphism),
we have a well-defined induced epimorphism $\bar{\psi} : \Gamma_n
\to \mathbb{Z}$. Let $\bar{\Gamma}_n =\ker \bar{\psi}$. Then $\bar{\Gamma}_n$
is a PTFA group, so $\mathbb{Z}\bar{\Gamma}_n$ has a right ring of quotients
$\mathbb{K}_n=(\mathbb{Z}\bar{\Gamma}_n)S_n ^{-1},$ where $S_n=\mathbb{Z}\bar{\Gamma}_n \setminus \{ 0 \}$.
Let $R_n:=(\mathbb{Z}\Gamma_n)S_n^{-1}.$ $R_n$ and $\mathcal{K}_n$ are flat left
$\mathbb{Z}\Gamma_n$-modules.

A very important role in what follows is played by the fact that $R_n$ is a PID; in fact, $R_n$ isomorphic to the ring of skew-Laurent
polynomials $\mathbb{K}_n[t^{\pm 1}]$. This can be seen as follows: by choosing a $t \in
\Gamma_n$ such that $\bar {\psi} (t)=1$, one obtains a splitting $\phi$
of $\bar{\psi}$, and the embedding $\mathbb{Z}\bar{\Gamma}_n \subset \mathbb{K}_n$
extends to an isomorphism $R_n \cong \mathbb{K}_n[t^{\pm 1}]$. However this
isomorphism depends in general on the choice of splitting of~$\bar\psi$.

\begin{definition}
\mbox{}
\begin{enumerate}
 \item The \emph{$n$-th order localized Alexander module} of the plane 
curve $C$ is defined to be $$\mathcal{A}_n(C)=H_1(U;R_n):= H_1(U;\mathbb{Z}\Gamma_n)\otimes_{\mathbb{Z}\Gamma_n}R_n,$$ viewed
as a right $R_n$-module.  The coefficients in the rightmost expression are the rank $1$ local system induced by the projection $G\twoheadrightarrow \Gamma_n$ \cite[section 5]{Harvey}. If we choose a splitting $\phi$ to
identify $R_n$ with $\mathbb{K}_n[t^{\pm 1}]$, we define
$\mathcal{A}^{\phi}_n(C)=H_1(U;\mathbb{K}_n[t^{\pm 1}])$.
 \item
The \emph{$n$-th order degree of $C$} is defined to be:
$$\delta_n(C)=\text{rk}_{\mathbb{K}_n} \mathcal{A}_n(C)=\text{rk}_{\mathbb{K}_n} \mathcal{A}^{\phi}_n(C).$$
\end{enumerate}
\end{definition}



The higher order degrees $\delta_n(C)$ are integral invariants of the fundamental
group $G$ of the complement (endowed with the linking number homomorphism). Indeed, by \cite{Harvey}, one has:
$$\delta_n(C)=\text{rk}_{\mathbb{K}_n} \left(
{G^{(n+1)}_r}/[G^{(n+1)}_r,G^{(n+1)}_r] \otimes_{\mathbb{Z}\bar{\Gamma}_n}
\mathbb{K}_n \right).$$
Note that since the isomorphism between $R_n$ and $\mathbb{K}_n[t^{\pm1}]$ depends on the choice of splitting, 
one cannot define a higher order version of the (univariable) Alexander polynomial in a canonical way. However, for any choice of splitting, 
the degree of the associated higher order Alexander polynomial is the same, hence this yields a well-defined invariant of $G$, which is exactly 
the higher order degree $\delta_n$ defined above.

\subsection{An effective method to compute $\delta_n(C)$}
\label{remFox}
The higher order degrees of $C$ may be computed by means of Fox free calculus from a presentation of $G=\pi_1(U)$, where $U:=\mathbb{C}^2\setminus C$ 
see \cite[Section 6]{Harvey} for details, although the computations can be quite tedious in practice. Such techniques will be used freely in this paper, as summarized in this section.

Consider the matrix of Fox derivatives for a presentation of $\pi_1(U)$ given by
$$
G=\pi_1(U)=\langle a_1,\ldots,a_m \text{\ \ \ } \vert\text{\ \ \ } r_j,\text{ \ }j=1,\ldots,l\rangle,
$$ 
that, is, the matrix
$$
\left(\frac{\partial r_j}{\partial a_i}\right)_{i,j} , \ \ \ 1\leq i \leq m, 1\leq j \leq l,
$$
which has entries in $\mathbb{Z}G$, and we take its involution (the $\mathbb{Z}$-linear map that takes elements of $G$ to their inverses)
$$
A=\left(\overline{\frac{\partial r_j}{\partial a_i}}\right)_{i,j} .
$$
Let $q_n:G\longrightarrow \Gamma_n$ be the projection, and let $q_n':\mathbb{Z}G\longrightarrow \mathbb{Z}\Gamma_n$ be the induced map on group rings. Let
$$
B(n)=A^{q_n'},
$$
that is, the matrix formed by the images of the entries of $A$ by $q_n'$. 

With this notation, $B(n)$ is a presentation matrix for the right $\mathbb{Z}\Gamma_n$-module $H_1(U, u_0;\mathbb{Z}\Gamma_n)$, where $u_0$ is some base point.

Moreover, since $R_n$ and $\mathcal{K}_n$ are flat over $\mathbb{Z}\Gamma_n$, we have that $B(n)$ is a presentation matrix for the right $R_n$-module (resp. $\mathcal{K}_n$-module) $H_1(U, u_0;R_n)$ (resp. $H_1(U, u_0;\mathcal{K}_n)$). 
By \cite[Proposition 5.6]{Harvey}, one obtains the following property for $B(n)$. 

\begin{lemma}\label{lem:Fox}
The rank of the left $\mathcal{K}_n$-module generated by the rows of $B(n)$ is $\leq m-1$, and
the rank of the left  $\mathcal{K}_n$-module generated by the rows of $B(n)$  is $m-1\Leftrightarrow \delta_n(C)$ is finite.
\end{lemma}

By doing allowable row and column operations to $B(n)$ in $R_n\cong\mathbb{K}_n[t^{\pm1}]$ (\cite[Lemma 9.2]{Harvey}), 
we can turn $B(n)$ into a different presentation matrix of $H_1(U, u_0;R_n)$ of the form
$$
\left(\begin{array}{c|c}
D & 0\\
\hline
\array{ccc} 0&...&0 \endarray& \array{ccc} 0&...&0 \endarray
\end{array}\right)
$$
where $D$ is a diagonal matrix with entries in $\mathbb{K}_n[t^{\pm1}]$ and 
the last row is a row of zeroes. 
The following result is immediate and allows one to obtain the invariant~$\delta_n(C)$. In it, the degree of a non-zero element of $\mathbb{K}_n[t^{\pm1}]$ is defined as the difference between the highest and lowest exponents of $t$ appearing in the polynomial.

\begin{proposition}\label{prop:Fox}
Under the conditions above, the higher order degree $\delta_n(C)$ is the degree of the product of the diagonal 
elements of $D$ if all of those elements are non-zero, and $\delta_n(C)=\infty$ otherwise.
\end{proposition}

\section{Vanishing of higher order degrees of transversal intersections.}

The goal of this section is to prove Theorem~\ref{thmCharac}, which characterizes the vanishing of the higher 
order degrees of a curve that is the union of two curves that intersect transversally and do not intersect at infinity.

\begin{remark}
The right hand side of the ``$\Leftrightarrow$'' equivalences in Theorem~\ref{thmCharac} is always satisfied in the cases described in Theorem~\ref{thmfinite}.
\label{remFinite}
\end{remark}

The proof of Theorem~\ref{thmCharac} is going to be broken down into $3$ lemmas (\ref{prop0},~\ref{thmTrivialAP}, and~\ref{thmTrans}). Before we prove those, let us write down some facts that will be used throughout the section.

\begin{proposition}[\cite{MaxLeidy}, Remark 3.3, Remark 3.9, Proposition 5.1]\label{propIrr}
If $C$ is an irreducible curve, then 
$$
\delta_0(C)=0\Longleftrightarrow G_r^{(1)}=G_r^{(2)}\Longleftrightarrow \delta_n(C)=0 \text{ \ \ for all \ \ }n\geq 0.
$$
\end{proposition}

\begin{remark}
There are three curves in the statement of Theorem~\ref{thmCharac}, namely $C$, $C'$ and $C''$. 
We will use $'$ or $''$ to refer to the objects corresponding to $C'$ and $C''$ respectively. 
For example, $U':=\mathbb{C}^2\backslash C'$, $G'':=\pi_1(U'')$, etc.
\label{remNot}
\end{remark}

\begin{theorem}[The Oka-Sakamoto theorem, \cite{OS}]\label{thmos}
Let $C=C'\cup C''\subset \mathbb{C}^2$ be the union of two affine plane curves, with $\deg C'=m'$ and $\deg C''=m''$. Suppose that $C'\cap C''$ consists on $m'm''$ distinct points in $\mathbb{C}^2$. Then,
$G\cong G'\times G''.$
\end{theorem}

\begin{remark}\label{rempresentation}
In the conditions of the Oka-Sakamoto theorem, we can consider a presentation for $G$ with generators $a_1,\ldots, a_{m'},b_1,\ldots,b_{m''}$, 
where the $a_i$'s are a choice of positively oriented meridians around $C'$ generating $G'$, and the $b_j$'s are a choice of positively oriented 
meridians around irreducible components of $C''$ generating $G''$ \cite{lib2}. 
Let $R'$ and $R''$ be a set of relations of a presentation of $G'$ and $G''$ where the generators are the $a$'s and $b$'s respectively. Then, we have 
the following presentation for $G$:
$$
G=\langle a_1,\ldots, a_{m'},b_1,\ldots,b_{m''} \vert [a_i,b_j]\text{ for all } i=1,\ldots,m'\text{ and}\text{\ } j=1,\ldots, m''; R'; R''\rangle.
$$
\end{remark}
The first of our three key lemmas deals with the $0$-th order degree of a union of two transversal irreducible curves.

\begin{lemma}\label{prop0}
Let $C=C'\cup C''\subset \mathbb{C}^2$ be the union of two irreducible affine plane curves, with $\deg C'=m'$ and $\deg C''=m''$. 
Suppose that $C'\cap C''$ consists on $m'm''$ distinct points in $\mathbb{C}^2$. Then,
$
\delta_0(C)=0.
$
\end{lemma}

\begin{proof}
By Theorem~\ref{thmos}, $G\cong G'\times G''$. We have that
$$
G_r^{(1)}/G_r^{(2)}\cong (G')_r^{(1)}/(G')_r^{(2)}\times(G'')_r^{(1)}/(G'')_r^{(2)},
$$

By \cite{Harvey}, one has:
$$\delta_n(C)=\text{rk}_{\mathbb{K}_n} \left(
{G^{(n+1)}_r}/[G^{(n+1)}_r,G^{(n+1)}_r] \otimes_{\mathbb{Z}\bar{\Gamma}_n}
\mathbb{K}_n \right).$$

Notice that the tensor product kills the $\mathbb{Z}$-torsion, so this is equivalent to
\begin{equation}\label{eq3} \delta_n(C)=\text{rk}_{\mathbb{K}_n} \left(
{G^{(n+1)}_r}/G^{(n+2)}_r \otimes_{\mathbb{Z}\bar{\Gamma}_n}
\mathbb{K}_n \right).\end{equation}

Note that $\mathbb{Z}\bar\Gamma_0\cong\mathbb{Z}[t^{\pm 1}]$ in this case. Since both $C'$ and $C''$ are irreducible, we have that $\Gamma'_0\cong \Gamma''_0\cong \mathbb{Z}$, $\bar{\Gamma}'_0\cong \bar{\Gamma}'_0$ are the trivial group, and $\mathbb{K}'_0\cong \mathbb{K}''_0\cong \mathbb{Q}$. By Proposition~\ref{propIrr}, $\delta_n$ of any irreducible curve is finite for all $n\geq 0$, so
$$
\delta_0(C')=\rk_{\mathbb{Q}} \left(
(G')_r^{(1)}/(G')_r^{(2)} \otimes_{\mathbb{Z}}
\mathbb{Q} \right)<\infty
$$
and the same statement holds for $C''$, which means that $(G')_r^{(1)}/(G')_r^{(2)}$ and \linebreak$(G'')_r^{(1)}/(G'')_r^{(2)}$ are both finite rank free abelian groups. Let us call them $A$ and $B$ for simplicity.

Now,
$$
\delta_0(C)=\text{rk}_{Q(\mathbb{Z}[t^{\pm 1}])} \left(
(A\oplus B) \otimes_{\mathbb{Z}[t^{\pm 1}]}Q(\mathbb{Z}[t^{\pm 1}]) \right),
$$
where $Q(\cdot)$ denotes taking the field of quotients.

Let $a\in A$, and let $k$ be an integer bigger than the rank of $A$. 
We have that $a, at,\ldots, at^k$ are linearly dependent, so $a$ is annihilated 
by some polynomial in $\mathbb{Z}[t^{\pm 1}]$. The same holds for all $b\in B$. Hence,
$
\delta_0(C)=0.
$
\end{proof}
The proofs of lemmas~\ref{thmTrivialAP} and~\ref{thmTrans} consist on applying the techniques of Section~\ref{remFox} to conveniently chosen presentations of the fundamental group.
\begin{lemma}\label{thmTrivialAP}
Let $n$ be a fixed integer, with $n\geq 0$. Let $C=C'\cup C''\subset \mathbb{C}^2$ be the union of two affine plane curves, with $\deg C'=m'$ and $\deg C''=m''$. Suppose that $C'\cap C''$ consists on $m'm''$ distinct points in $\mathbb{C}^2$. Suppose that $C'$ is irreducible, with $\delta_0(C')=0$. Then, $\delta_n(C)$ is finite, and
$\delta_n(C)=0 \Leftrightarrow \delta_n(C'')<\infty.$
\end{lemma}

\begin{proof}
By Theorem~\ref{thmos}, $G\cong G'\times G''$. We first consider the case where $C''$ is also an irreducible curve such that $\delta_0(C'')=0$. In this situation, we know that $\delta_n(C'')=0$ for all $n\geq 0$, and, in fact, the stronger statement $(G'')_r^{(1)}=(G'')_r^{(2)}$ holds (Proposition~\ref{propIrr}). Since $G$ is the direct product of $G'$ and $G''$, we have that
$$
G_r^{(n+1)}/G_r^{(n+2)}\cong (G')_r^{(n+1)}/(G')_r^{(n+2)}\times(G'')_r^{(n+1)}/(G'')_r^{(n+2)},
$$
which is the trivial group for all $n\geq 0$. By equation~\eqref{eq3}, one gets that $\delta_n(C)=0$ for all $n\geq 0$.

From now on, we assume that $C''$ is either not irreducible, or if it is irreducible, then $\delta_0(C'')\neq 0$. 

We consider the presentation of $G$ described in Remark~\ref{rempresentation}. Since $(G')_r^{(1)}=(G')_r^{(2)}$ 
(Proposition~\ref{propIrr}), $a_ia_k^{-1}=1$ in $\mathbb{Z}\Gamma_n$ for all $i,k\in \{1,\ldots, m'\}$, $n\geq 0$.

From now on $n$ is some integer such that $n\geq 1$ if $C''$ is an irreducible curve with $\delta_0(C'')\neq 0$, 
and $n\geq 0$ if $C''$ is not irreducible. Note that, if $C''$ is irreducible, the result for $n=0$ is already 
proved in Lemma~\ref{prop0}.

Let $x_1=a_1$, and $x_i=a_ia_1^{-1}$ for all $i=2,\ldots, m'$. Let $y_j=b_ja_1^{-1}$ for all $j=1,\ldots,m''$. 
We obtain the presentation
$$
G= \left\langle 
x_1,\ldots, x_{m'},y_1,\ldots,y_{m''} \mid
\array{l}
[x_1,y_j],\ x_ix_1y_jx_i^{-1}x_1^{-1}y_j^{-1},\ \widetilde{R}',\ \widetilde{R}''\\
i=2,\ldots, m',\ j=1,\ldots, m''\\
\endarray
\right\rangle
$$
where $\widetilde{R}'$ are some relations in $x_1, \ldots, x_{m'}$, and $\widetilde{R}''$ are the same relations as $R''$ if we switch the letter $b_j$ for $y_j$, for all $j=1,\ldots, m''$. Indeed, if we plug in $y_jx_1$ for $b_j$ in the relations $R''$, the $x_1$'s cancel out because they commute with all the $y_j$'s and because the linking number homomorphism takes any word in the $b$ letters to the sum of the exponents appearing on that word, so the sum of the exponents of words in $R''$ must be zero.

We may assume by reordering that $y_1\neq y_2$ in $\Gamma_n$, where $n\geq 1$ if $C''$ is irreducible, and $n\geq 0$ otherwise. Let us see this. Indeed, if $C''$ is not irreducible, this amounts to $b_1$ and $b_2$ being positively oriented meridians around different irreducible components of $C''$, which we can assume after reordering. If $C''$ is irreducible but $\delta_0(C'')\neq 0$, Proposition~\ref{propIrr} says that $(G'')_r^{(2)}\subsetneqq (G'')_r^{(1)}$, which implies that there exist $j\neq l$ in $\{1,\ldots, m''\}$ such that $b_jb_l^{-1}\neq 1$ in $\Gamma''_1$. Reordering, we may assume that $j=1$ and $l=2$, and we get that $y_1\neq y_2$ in $\Gamma_n$.

Consider the involution of the matrix of Fox derivatives for this presentation of $G$ with coefficients in $\mathbb{Z}\Gamma_n$ ($B(n)$ in the notation of Section~\ref{remFox}),
\small
$$
\left(\begin{array}{c|c|c|c|c|c}
\text{\----}\ (1-y_j^{-1})\ \text{\----}& \text{\----}\ (x_2^{-1}-y_j^{-1})\ \text{\----}&\cdots &\text{\----}\  (x_{m'}^{-1}-y_j^{-1})\ \text{\----}& \ & \ \\
\text{\----}\ 0\ \text{\----} &\text{\----}\ (1-x_1^{-1}y_j^{-1})\ \text{\----}&\cdots &\text{\----}\  0\ \text{\----}& \ & \ \\
\text{\----}\ 0\ \text{\----} &\text{\----}\ 0\ \text{\----} &\cdots & \text{\----}\ 0\ \text{\----}  & A' & 0 \\
\vdots & \vdots  & \vdots & \vdots  & \ & \ \\
\text{\----}\ 0\ \text{\----} &\text{\----}\ 0\ \text{\----} &\cdots &\text{\----}\  (1-x_1^{-1}y_j^{-1})\ \text{\----} & \ & \ \\
\hline
\ & \ & \ & \ & \ & \ \\
(x_1^{-1}-1)I_{m''}& (x_1^{-1}x_2^{-1}-1)I_{m''}& \cdots & (x_1^{-1}x_{m'}^{-1}-1)I_{m''}& 0 & A''\\
\end{array}\right),
$$
\normalsize
where ``$\text{\----}\ z_j\ \text{\----}$'' denotes a row of $m''$ elements whose $j$-th entry is $z_j$ for $j=1,\ldots, m''$, and $I_{m''}$ is the identity matrix of dimension $m''\times m''$. $A'$ is the matrix corresponding to the relations $\tilde{R}'$, and $A''$ is the matrix that computes $\delta_n(C'')$ with coefficients in $\mathbb{Z}\Gamma_n''$, which is identified with $\mathbb{Z}\bar \Gamma_n$ by the isomorphism of groups
$h:\Gamma_n''\to \bar \Gamma_n$, given by $h(b_j)=y_j.$

First, note that $x_i=1$ in $\Gamma_n$, for all $i=2, \ldots, m'$.
%
In addition, the left part of this matrix consists on $m'$ blocks of dimensions $(m'+ m'')\times m''$. 
We subtract the $i$-th column to the $i$-th column of the $j$-th block, for all $i=1,\ldots, m''$, $j=2,\ldots, m'$, to get
\small
\begin{equation}
\label{eqnpreunit}
\left(\begin{array}{c|c|c|c|c|c}
\text{\----}\ (1-y_j^{-1})\ \text{\----}& \text{\----}\ 0\ \text{\----}&\cdots &\text{\----}\  0\ \text{\----}& \ & \ \\
\text{\----}\ 0\ \text{\----} &\text{\----}\ (1-x_1^{-1}y_j^{-1})\ \text{\----}&\cdots &\text{\----}\  0\ \text{\----}& \ & \ \\
\text{\----}\ 0\ \text{\----} &\text{\----}\ 0\ \text{\----} &\cdots & \text{\----}\ 0\ \text{\----}  & A' & 0 \\
\vdots & \vdots  & \vdots & \vdots  & \ & \ \\
\text{\----}\ 0\ \text{\----} &\text{\----}\ 0\ \text{\----} &\cdots &\text{\----}\  (1-x_1^{-1}y_j^{-1})\ \text{\----} & \ & \ \\
\hline
\ & \ & \ & \ & \ & \ \\
(x_1^{-1}-1)I_{m''}& 0& \cdots & 0 & 0 & A''\\
\end{array}\right).
\end{equation}
\normalsize

Note that $1-y_j^{-1}\neq 0$ in $\mathbb{Z}\bar\Gamma_n$ for any $j=1,\ldots, m''$, since $1-y_j^{-1}\neq 0$ in $\mathbb{Z}\bar\Gamma_0$. 
We multiply row $m'+1$ by $1-y_1^{-1}$ on the left, and add to it the first row times $1-x_1^{-1}$, and the ($m'+j$)-th row times $1-y_j^{-1}$ for all $j=2,\ldots, m''$, to get 

\footnotesize
$$
\left(\begin{array}{c|c|c|c|c|c|c}
\multicolumn{2}{c|}{\text{\----}\ (1-y_j^{-1})\ \text{\----}}& \text{\----}\ 0\ \text{\----}&\cdots &\text{\----}\ 0\ \text{\----}& \ & \ \\
\multicolumn{2}{c|}{\text{\----}\ 0\ \text{\----} }&\text{\----}\ (1-x_1^{-1}y_j^{-1})\ \text{\----}&\cdots &\text{\----}\ 0\ \text{\----}& \ & \ \\
\multicolumn{2}{c|}{\text{\----}\ 0\ \text{\----} }&\text{\----}\ 0\ \text{\----} &\cdots & \text{\----}\ 0\ \text{\----}  & A' & 0 \\
\multicolumn{2}{c|}{\vdots} & \vdots  & \vdots & \vdots  & \ & \ \\
\multicolumn{2}{c|}{\text{\----}\ 0\ \text{\----} }&\text{\----}\ 0\ \text{\----} &\cdots &\text{\----}\  (1-x_1^{-1}y_j^{-1})\ \text{\----} & \ & \ \\
\hline
\multicolumn{2}{c|}{\text{\----}\ 0\ \text{\----} }&\text{\----}\ 0\ \text{\----} &\cdots & \text{\----}\ 0\ \text{\----} & (1-x_1^{-1})a'_{1,*} & \ \\
\cline{1-6}
\vert & \ & \ & \ & \ & \ \\
0 &(x_1^{-1}-1)I_{m''}& 0& \cdots & 0& 0 & A''\\
\vert & \ & \ & \ & \ & \ \\
\end{array}\right),
$$
\normalsize
where $a'_{1,*}$ is the first row of $A'$, so its entries are polynomials in $\mathbb{Z}[x_1]$, which commute with elements of $\mathbb{K}_n$, 
which is identified by $h$ with $\mathcal{K}_n''$.

We now focus on the second to $m'$-th blocks of size $m'\times m''$ at the top of the matrix. We can multiply the $j$-th column (on the right) 
of each of these blocks by $y_j$ for all $j=1,\ldots, m''$, and subtract the second from the first column of each of these blocks to get 
$y_1-y_2$ as the first entry and $y_j-x_1^{-1}$  as the $j$-th entry of the $k$-th row of the $k$-th block, where $k=2,\ldots, m'$, $j=2,\ldots, m'$. 
Note that $y_1\neq y_2$ in $\mathbb{Z}\bar \Gamma_n$, so $y_1-y_2$ has an inverse in $R_n$. Now, we multiply the first column (on the right) 
by the inverse of $1-y_1^{-1}$, and the first column of the $j$-th block of size $m'\times m''$ by the inverse of $y_1-y_2$ for all $j=2,\ldots,m''$. 
Reordering the columns, putting the ones corresponding to the first column of every $m'\times m''$ block first, we get
\begin{equation}
\scriptscriptstyle
\left(\begin{array}{ccc|ccc|ccc|c|ccc}
\multicolumn{3}{c|}{I_{m'}} & \multicolumn{3}{c|}{0} & \multicolumn{3}{c|}{*}& A' & \multicolumn{3}{c}{*}\\
\hline
\multicolumn{3}{c|}{\ }&\multicolumn{3}{c|}{ \ }  & \multicolumn{3}{c|}{} & (1-x_1^{-1})a'_{1,*}  & \multicolumn{3}{c}{\ }\\
\cline{10-10}
\multicolumn{3}{c|}{0 }&\multicolumn{3}{c|}{ A''}  & \multicolumn{3}{c|}{ B} & 0  & \multicolumn{3}{c}{0}\\
\end{array}\right),
\label{eqnClose}
\end{equation}
where $B$ is the matrix
$$
\left(\begin{array}{c}
 \text{\----}\ 0\ \text{\----}\\
\hline
\ \\
(x_1^{-1}-1)I_{m''-1} \\
\ \\
\end{array}\right).
$$

Hence, performing column operations we can turn matrix~\eqref{eqnClose} into
\begin{equation}
\scriptscriptstyle
\left(\begin{array}{ccc|ccc|ccc|c|ccc}
\multicolumn{3}{c|}{I_{m'}} & \multicolumn{3}{c|}{0} & \multicolumn{3}{c|}{0}& 0 & \multicolumn{3}{c}{0}\\
\hline
\multicolumn{3}{c|}{\ }&\multicolumn{3}{c|}{ \ }  & \multicolumn{3}{c|}{} & (1-x_1^{-1})a'_{1,*}  & \multicolumn{3}{c}{\ }\\
\cline{10-10}
\multicolumn{3}{c|}{0 }&\multicolumn{3}{c|}{ A''}  & \multicolumn{3}{c|}{ B} & 0  & \multicolumn{3}{c}{0}\\
\end{array}\right),
\label{eqnClose2}
\end{equation}

Let $k$ be the rank of the left $\mathcal{K}_n''$-module spanned by the rows of $A''$. By Proposition~\ref{prop:Fox}, 
$k$ is at most $m''-1$, and $k$ is equal to $m''-1$ if and only if $\delta_n(C'')<\infty$. Identifying 
$\mathcal{K}_n''$ with $\mathbb{K}_n$ by $h$, we get that the rank of the left $\mathbb{K}_n$-module spanned by the rows 
of $A''$ is $k$ as well. Hence, doing row and column operations in $\mathcal{K}_n''$, and noting that $x_1$ commutes 
with $\mathcal{K}_n''$ in $R_n$, we can turn the matrix~\eqref{eqnClose2} into
$$
\scriptscriptstyle
\left(\begin{array}{ccc|ccc|ccc|ccc|ccc}
\multicolumn{3}{c|}{I_{m'}} & \multicolumn{3}{c|}{0} & \multicolumn{3}{c|}{0}& \multicolumn{3}{c}{0 }& \multicolumn{3}{c}{0 }\\
\hline
\multicolumn{3}{c|}{0 }&\multicolumn{3}{c|}{I_{k}}& \multicolumn{3}{c|}{\ }  &\multicolumn{3}{c|}{\ }& \multicolumn{3}{c}{\  }\\
\cline{1-6}
\multicolumn{3}{c|}{0 }&\multicolumn{3}{c|}{0}& \multicolumn{3}{c|}{(x_1^{-1}-1)\widetilde B } & \multicolumn{3}{c|}{(x_1^{-1}-1)E }  & \multicolumn{3}{c}{0}\\
\end{array}\right),
$$
where $\widetilde B$ is an $m''\times m''$ matrix in $\mathcal{K}''_n$ such that the rank of the left $\mathcal{K}''_n$-module spanned by its rows is $m''-1$, and $E$ is a matrix with entries in $R_n$. In particular, the rank of the left $\mathcal{K}''_n$-module spanned by the last $m''-k$ rows of $\widetilde B$  is greater or equal than $m''-k-1$, and at most $m''-k$. Let us denote by $D$ and $F$ the matrices formed by the last $m''-k$ rows of $\widetilde B$ and $E$ respectively. Doing column operations, we get
\begin{equation}
\scriptscriptstyle
\left(\begin{array}{ccc|ccc|ccc|ccc|ccc}
\multicolumn{3}{c|}{I_{m'}} & \multicolumn{3}{c|}{0} & \multicolumn{3}{c|}{0}& \multicolumn{3}{c|}{0}& \multicolumn{3}{c}{0}\\
\hline
\multicolumn{3}{c|}{0 }&\multicolumn{3}{c|}{I_{k}}& \multicolumn{3}{c|}{0 }  &\multicolumn{3}{c|}{0}&\multicolumn{3}{c}{0}\\
\hline
\multicolumn{3}{c|}{0 }&\multicolumn{3}{c|}{0}& \multicolumn{3}{c|}{(x_1^{-1}-1)D } & \multicolumn{3}{c|}{(x_1^{-1}-1)F} & \multicolumn{3}{c}{0}\\
\end{array}\right).
\label{eqnE}
\end{equation}

By Lemma~\ref{lem:Fox}, the rank of the left $\mathcal{K}_n$-module generated by the rows of this matrix should be less 
or equal than $m'+m''-1$, which rules out the possibility of the rank of the left $\mathcal{K}_n''$-module spanned by 
the rows of $D$ being $m''-k$. Hence, the rank of the left $\mathcal{K}_n''$-module spanned by the rows of $D$ is 
$m''-k-1$. If we keep doing row and column operations to $D$ in $\mathcal{K}''_n$, and perhaps permuting some of 
the last $m''-k$ rows of matrix~\eqref{eqnE} at the end, one obtains
$$
\scriptscriptstyle
\left(\begin{array}{c|c|c|c}
I_{m'}& 0 & 0 &0\\
\hline
0 &I_{k}& 0 &0\\
\hline
0 &0& (x_1^{-1}-1)I_{m''-k-1} & \ \\
\cline{1-3}
\text{\----}\ 0\ \text{\----}& \text{\----}\ 0\ \text{\----}& \text{\----}\ 0\ \text{\----} &(1-x_1^{-1})* 
\end{array}\right),
$$
where $*$ is a matrix in $R_n$. But again, by rank considerations, the last row of this matrix must be identically $0$. 
Performing column operations, we can turn $(1-x_1^{-1})*$ into the zero matrix. Hence, $\delta_n(C)=m''-k-1$, which 
is a finite number. This means that $\delta_n(C)=0$ if and only if $\delta_n(C'')$ is finite.
\end{proof}

\begin{lemma}\label{thmTrans}
Let $C=C'\cup C''\subset \mathbb{C}^2$ be the union of two affine plane curves, with $\deg C'=m'$ and $\deg C''=m''$. 
Suppose that $C'\cap C''$ consists on $m'm''$ distinct points in $\mathbb{C}^2$. Suppose that neither $C'$ nor $C''$ are irreducible with $\delta_0=0$.
Then $$\delta_n(C)=0 \text{ for all } n\geq 1.$$
Moreover, if either both $C'$ and $C''$ are not irreducible, or both irreducible, the equality holds for all $n\geq 0$. 
However, if one of the curves, say $C'$, is irreducible and the other one ($C''$) is not, then $\delta_0(C)$ is finite, and
$$
\delta_0(C)=0\Leftrightarrow \delta_0(C'')<\infty.
$$
\end{lemma}

\begin{proof}
If both $C'$ and $C''$ are irreducible, the result for $n=0$ follows from Lemma~\ref{prop0}.

We consider the presentation of $G$ described in Remark~\ref{rempresentation}. 
If $C'$ is not irreducible, we can assume that $a_1\neq a_2$ in $\Gamma_0$ by reordering, so $a_2a_1^{-1}\neq 1$ 
in $\mathbb{Z}\bar\Gamma_n$ for any $n\geq 0$. Similarly, if $C''$ is not irreducible, we can assume that 
$b_2b_1^{-1}\neq 1$ in $\mathbb{Z}\bar\Gamma_n$ for any $n\geq 0$. After reordering, we may assume the same 
condition if $C''$ is irreducible with $\delta_0(C'')\neq 0$ (resp. $C'$), but this time for $n\geq 1$, as 
justified in the proof of Lemma~\ref{thmTrivialAP}.


We deal with the case when $n$ is some integer greater or equal than $1$ if either $C'$ or $C''$ are irreducible, 
and $\geq 0$ otherwise.

Let $x_1=a_1$, and $x_i=a_ia_1^{-1}$ for all $i=2,\ldots, m'$. Let $y_1=b_1$, and $y_j=b_jb_1^{-1}$ for all $j=2,\ldots,m''$. 
We obtain the presentation
$$
G=\langle x_1,\ldots, x_{m'},y_1,\ldots,y_{m''} \vert [x_i,y_j]\text{ for all } i=1,\ldots,m'\text{ and}\text{\ } j=1,\ldots, m''; \widetilde{R}'; \widetilde{R}''\rangle
$$
where $\widetilde{R}'$ (resp. $\widetilde{R}''$) are the defining relations for $\pi_1(\mathbb{C}^2\setminus C')$ 
(resp. $\pi_1(\mathbb{C}^2\setminus C'')$) in $x_1, \ldots, x_{m'}$ (resp. $y_1,\ldots,y_{m''}$).

Consider the matrix $B(n)$ described in Section~\ref{remFox}, that is,
\small
\begin{equation}
\left(\begin{array}{c|c|c|c|c}
\text{\----}\ (1-y_j^{-1})\ \text{\----} & \text{\----}\ 0\ \text{\----} &\cdots & \text{\----}\ 0\ \text{\----} &\ \\
\text{\----}\ 0\ \text{\----} &\text{\----}\ (1-y_j^{-1})\ \text{\----} &\cdots & \text{\----}\ 0\ \text{\----}&\ \\
\vdots & \vdots  &\cdots & \vdots  & *\\
\text{\----}\ 0\ \text{\----} & \text{\----}\ 0\ \text{\----} &\cdots &\text{\----}\ (1-y_j^{-1})\ \text{\----} &\ \\
\hline
(x_1^{-1}-1)I_{m''} & (x_2^{-1}-1)I_{m''} &\cdots & (x_{m'}^{-1}-1)I_{m''}&*\\
\end{array}\right),
\label{eqn1}
\end{equation}
\normalsize
where the rightmost columns correspond to $\widetilde{R}'$ and $\widetilde{R}''$.

Note that $x_2^{-1}-1$ is non-zero in $\mathbb{Z}\bar\Gamma_n$ if $n\geq 1$, and, if $C'$ is not irreducible, 
also for $n=0$. We begin by multiplying the last $m''$ rows by the inverse of $x_2^{-1}-1$ (on the left), and then, 
by performing column operations, one obtains
\small
$$
\left(\begin{array}{c|c|c|c|c}
\text{\----}\ (1-y_j^{-1})\ \text{\----} & \text{\----}\ 0\ \text{\----} &\cdots & \text{\----}\ 0\ \text{\----} &\ \\
\text{\----}\ *\ \text{\----} &\text{\----}\ (1-y_j^{-1})\ \text{\----} &\cdots & \text{\----}\ *\ \text{\----}&\ \\
\vdots & \vdots  &\cdots & \vdots  & *\\
\text{\----}\ 0\ \text{\----} & \text{\----}\ 0\ \text{\----} &\cdots &\text{\----}\ (1-y_j^{-1})\ \text{\----} &\ \\
\hline
0& I_{m''} &\cdots & 0&0\\
\end{array}\right).$$
\normalsize
Hence, we may compute $\delta_n(C)$ using the matrix formed by the first $m'$ rows of the matrix above 
without the columns of the second block. This new matrix consists on $m'-1$ blocks of $m'\times m''$ 
matrices, plus another matrix at the end, represented by the rightmost submatrix after the last vertical 
line. One can permute the first and second rows in this new matrix of $m'$ rows, 
and then permute columns so that the first $m'$ columns of the resulting matrix are the second columns 
of each of the first $(m'-1)$ blocks of size $m'\times m''$. This way one obtains a matrix of the form,
$$
\left(\begin{array}{c|c}
*  & * \\
\hline
(1-y_2^{-1})I_{m'-1}& *\\
\end{array}\right).
$$
Note that $1-y_2^{-1}$ is non-zero in $\mathbb{Z}\bar \Gamma_n$ for $n\geq 1$, and, if $C''$ is not 
irreducible, also in $\mathbb{Z}\bar \Gamma_0$. Finally, multiplying each row on the left by the inverse 
of $1-y_2^{-1}$, and performing column and row operations, we see that $\delta_n(C)=0$.

Lastly, we consider the case when $n=0$, $C'$ is irreducible, and $C''$ is not irreducible. The proof of 
this is done by considering the same presentation for $G$ as the one explained in the proof of 
Lemma~\ref{thmTrivialAP}, and following the same computations done there, the only difference being that $n=0$ 
in this case and that, since $C''$ is not irreducible, we know that $y_1\neq y_2$ in $\mathbb{Z}\bar \Gamma_0$. 
Note that $x_i=1$ in $\mathbb{Z}\Gamma_0$ for $i=2,\ldots, m'$ because $C'$ is irreducible. Using the same 
notation as in the proof of Lemma~\ref{thmTrivialAP}, it follows that $\delta_0(C)=m''-k-1$ (a finite number), 
where $k$ is the rank of the left $\mathcal{K}_n''$-module spanned by the rows of $A''$, and $\delta_0(C'')$ 
is finite if and only if $k=m''-1$. This means that $\delta_0(C)=0$ if and only if $\delta_0(C'')$ is finite.
\end{proof}

\begin{example}\label{exmcubiclines}
Let $C'$ be an irreducible curve such that $\delta_0(C')\neq 0$. For example, we can take $C'$ to be the cuspidal 
cubic, which has $\delta_0(C')=2$, and $\delta_n(C')=1$ for $n\geq 1$ (\cite[Example 9.8]{Suky}). Let $C''$ be a 
collection of $m''$ parallel lines, each of which intersects $C'$ in three distinct points. Let $C=C'\cup C''$. 
Then, following the proof and notations of Theorem~\ref{thmTrans}, it follows that
$$
\left\{\begin{array}{lc}
\delta_0(C)=m''-1 & \ \\
\delta_n(C)=0 & \text{for all }n\geq 1.\\
\end{array}\right.
$$

Indeed, in this case the fundamental group of the complement to $m''$ parallel lines is the free group on $m''$ 
generators, and thus it has a presentation with no relations. Hence $A''$ is the empty matrix, so the $k$ that 
appears at the end of the proof of Theorem~\ref{thmTrans} is $0$.

This example shows that $\delta_0$ and $\delta_n$ can differ by arbitrarily large numbers, for $n\geq 1$. 
This cannot happen in the case of knots~(\cite{Cochran}), where $\delta_0\leq \delta_1+1\leq \delta_2 +1 \leq \ldots$.
\end{example}

\section{Restatement of the main theorem in terms of multivariable Alexander polynomials}
We start by recalling the relationship between the Alexander polynomials of a plane curve $C$ and $\delta_0(C)$. If $C$ is irreducible, 
the result below appears in \cite[Remark 3.9]{MaxLeidy}, and the non-irreducible case was done in~\cite[Theorem 11]{EvaMax}.

\begin{theorem}\label{thmAAP}
Let $C\subset \mathbb{C}^2$ be a plane curve with $s$ irreducible components. Then,
$$\delta_0(C)=\deg \Delta_C^{\multi}(t_1,\ldots,t_s)$$
\end{theorem}

\begin{remark}
We are using the convention $\deg 0=\infty$. The proof of \cite[Theorem 11]{EvaMax} assumes $\delta_0(C)$ is finite, 
but the result is also true for $\delta_0(C)=\infty$ because
$$
\delta_0(C)=\infty \Leftrightarrow F_1\left(H_1(U,u_0;R_0)\right)=0 \Leftrightarrow F_1\left(H_1(U,u_0;\mathbb{Z}\Gamma_0)\right)=0\Leftrightarrow \Delta_C^{\multi}=0.
$$
In this list of equivalences, we have used that the projection $G\twoheadrightarrow \Gamma_0$ is the abelianization morphism and that $R_0$ is flat as a $\mathbb{Z}\Gamma_0$-module.
\end{remark}

\begin{remark}\label{remconstant}
Theorem~\ref{thmAAP} tells us that $\delta_0(C)=0$ if and only if $\Delta_C^{\multi}$ has a representative which 
is a non-zero homogeneous polynomial, which, under the hypotheses of Theorem~\ref{thmCharac}, implies that 
$\Delta_C^{\multi}$ is a non-zero constant. Indeed, $\Delta_C^{\multi}$ is the $\gcd$ of the codimension-one 
minors of $B(0)$, and the matrix $B(0)$ in equation~\eqref{eqn1} has a codimension-one minor which does not 
have non-constant homogeneous factors.
\end{remark}

Now, we characterize the plane curves with infinite $\delta_0(C)$.

\begin{lemma}\label{lempencil}
Let $C$ be a plane curve, with $s$ irreducible components $C_1,\ldots, C_s$. Then, the following are equivalent.
\begin{enumerate}
\item $\delta_0(C)=\infty$.
\item\label{pencil} $s\geq 2$ and $C_i$ is the zero set of a polynomial of the form $f(x,y)+\lambda_i$ for all 
$1\leq i \leq s$, where $f(x,y)\in\mathbb{C}[x,y]$ is a polynomial of degree $d\geq 1$, and $\lambda_i\in\mathbb{C}$ 
for all $1\leq i \leq s$.
\item There exists an epimorphism $G\twoheadrightarrow \mathbb{F}_s$ onto the free group of rank $s\geq 2$.
\end{enumerate}
\end{lemma}

We will refer to condition~\eqref{pencil} as $C$ being of \emph{affine pencil type}.

\begin{proof}
Note that by Theorem~\ref{thmfinite}, $\delta_0(C)=\infty\Rightarrow s\geq 2$.

With the notation of Section~\ref{remFox}, $B(0)$ is a presentation matrix for $H_1(U, u_0;\mathbb{Z} \Gamma_0)$. By \cite[Proposition 3, Remark 12]{EvaMax}, we have the following result about the first homology jump loci \cite[Definition 9]{EvaMax}, which we'll use as an intermediate equivalent condition in our proof:
$$
\mathcal{V}_1(U)=(\mathbb{C}^*)^s\Leftrightarrow \text{ all the codimension }1\text{ minors of }B(0)\text{ are 0}.
$$
This last condition is equivalent to the rank of the left $\mathcal{K}_0$-module generated by the rows of 
$B(0)$ being strictly smaller than $m-1$, which by Remark~\ref{remFox} is equivalent to $\delta_0(C)=\infty$.

Let $\overline C$ be the projective completion of $C$, and let $D$ be the curve in $\mathbb{P}^2$ defined by 
$D=\bar C\cup L_\infty$, where $L_\infty$ is the line at infinity. By \cite[Theorem 4.1]{dimca2010}, the condition 
$\mathcal{V}_1(U)=(\mathbb{C}^*)^s$ (which can be reformulated in terms of cohomology jump loci 
by~\cite[p.50, (2.1)]{dimca}) is equivalent to the existence of a primitive pencil 
$\mathcal{C}_{[\alpha_1:\alpha_2]}=\alpha_1 P_1(x,y,z) +\alpha_2 P_2(x,y,z)$ of plane curves on $\mathbb{P}^2$  
having $s+1$ fibers (corresponding to $s+1$ different $[\alpha_1:\alpha_2]\in\mathbb{P}^1$) whose reduced support 
form a partition of the set of $s+1$ irreducible components of $D$. Hence, the reduced support of those $s+1$ 
fibers must be in one to one correspondence with the irreducible components of $D$, so we may write the pencil 
in the form $\beta_1 F(x,y,z)+\beta_2 z^d$, where $F(x,y,z)$ is a degree $d$ irreducible polynomial in 
$\mathbb{C}[x,y,z]$ and $[\beta_1:\beta_2]\in\mathbb{P}^1$. Restricting to the affine part (making $z=1$) yields (1)$\Leftrightarrow$(2).

For (2)$\Rightarrow$(3), we see that $f(x,y)$ induces the desired epimorphism in fundamental groups. 
By~\cite[Lemma 1.2.1]{Libgober-characteristic} (3) implies $\mathcal V_1(U)=(\mathbb C^*)^s$ and hence the argument 
above implies (2). 
\end{proof}


As a corollary of the lemma above, Theorem~\ref{thmAAP}, Remark~\ref{remconstant} and Theorem~\ref{thmCharac}, 
one obtains Corollary~\ref{corAP}, whose proof is below.

\begin{proof}[Proof of Corollary~\ref{corAP}]
The only thing left to show is the last statement. $\Delta_C^{\multi}$ can be computed with the matrix from
equation~\eqref{eqnpreunit}, as the operations we did to get from $B(0)$ to that matrix were all allowed in 
$\mathbb{Z}\Gamma_0\cong\mathbb{Z}[t^{\pm 1},t_1^{\pm 1}, \ldots, t_s^{\pm 1}]$. The abelianization morphism 
identifies $x_1$ with $t$ and $y_j$ with $\frac{t_j}{t}$ in equation~\eqref{eqnpreunit}, and with those identifications 
and up to multiplication by a unit in $\mathbb{Z}[t^{\pm 1},t_1^{\pm 1}, \ldots, t_s^{\pm 1}]$, $(t_j-1)^{m'-1}(1-t)^{m''}$ 
and $(t_j-t)(t_j-1)(1-t)^{m''-1}$ are $(m'+m''-1)$-minors of~\eqref{eqnpreunit} for all $j=1,\ldots, m''$. Hence, 
$\Delta_C^{\multi}$ divides the greatest common divisor of all these minors, so $\Delta_C^{\multi}$ divides 
$(t-1)^{m''-1}$. The equality can be achieved, \cite[Theorem 9.15, case (ii)]{suciu} gives an example where 
$\Delta_C^{\multi}(t,t_1,\ldots,t_s)=(t-1)^{m''-1}$.
\end{proof}

\begin{remark}\label{remcomputations}
The higher order degrees depend 
on the linking number homomorphism, which, if the curve is not irreducible, is not an invariant of the fundamental group of the curve complement. However, the multivariable Alexander polynomial only depends on the fundamental group (up to a change of basis in the variables). Thus, Corollary~\ref{corAP} gives us direct restrictions for which groups can be realized as fundamental groups of the complement of a union of transversal plane curves, and those restrictions can be computed from a presentation of the group.
\end{remark}

\begin{example}
As an example of a curve of affine pencil type one can consider the cuspidal cubic $f(x,y)=y^2-x^3$ and the curve 
$C''=\{(x,y)\in \mathbb C^2 \mid f(f-1)=0\}$ of degree $m''=6$. One can check that 
$$\pi_1(\mathbb C^2\setminus C'')=\langle \alpha_1,\alpha_2,\gamma : [\gamma,\alpha_2]=[\gamma^{\alpha_1},\alpha_2]=1,\gamma=\gamma^{\alpha_1}\gamma^{\alpha_1^{-1}}\rangle.$$
Here, $\alpha_1$ and $\gamma\alpha_1$ are positively oriented meridians about $f=0$, and $\alpha_2$ is a positively oriented meridian about $f=1$.

Let $r_n$ be the rank of the left $\mathcal{K}''_n$-module generated by the rows of the $3\times 3$ 
matrix $B(n)$ of Section~\ref{remFox} computed from the presentation of $\pi_1(\mathbb C^2\setminus C'')$ above. 
One has $\gamma^{-1}-1\in \mathbb{Z}\bar{\Gamma}_n''$ is $0$ if $n=0$ and a unit otherwise.  Using this, it is 
straightforward to check that $r_0=1$, and $r_n=2$ for all $n\geq 1$. By Lemma~\ref{lem:Fox}, $\delta_0(C'')=\infty$ and 
$\delta_n(C'')<\infty$ for all $n\geq 1$. In fact, using the methods of Section~\ref{remFox}, one can check 
that $\delta_n(C'')=0$ for all $n\geq 1$. Indeed,

\begin{align*}
B(n)=&
\left[
\array{ccc}
0 & -u\alpha_1 v & u\alpha_1+\alpha_1^{-1}u\alpha_1-u\\
-u & \alpha_1^{-1}u\alpha_1 & 0\\
v & -\alpha_1 v & \alpha_1^{-2}\gamma^{-1}\alpha_1-1+\alpha_1
\endarray
\right]
\cong
\left[
\array{ccc}
0 &\alpha_1 v & 1-w-\alpha_1 \\
-u & uw & 0 \\
0 & vw & \alpha_1^{-2}\gamma^{-1}\alpha_1-w 
\endarray
\right]\\
\cong &
\left[
\array{cc}
\alpha_1v & 1-w-\alpha_1\\
0 & w\alpha_1^{-1}w-w\alpha_1^{-1}+\alpha_1^{-2}\gamma^{-1}\alpha_1
\endarray
\right]
\cong 
\left[
\array{cc}
\alpha_1v & 1-w-\alpha_1\alpha_2^{-1}\\
0 & 0
\endarray
\right]\\
\end{align*}
where $u=(\gamma^{-1}-1)$, $v=(\alpha_2^{-1}-1)$, and $w=u^{-1}\alpha_1^{-1}u\alpha_1=[u^{-1},\alpha_1^{-1}]$. The first transformation is a result
of multiplying the first row by $-u^{-1}$ (on the left), and then adding row 1 and $vu^{-1}$ times row 2 to row 3. The second transformation results from eliminating column 1 and row 2
(since $u$ is a unit) and substracting $w\alpha_1^{-1}$ times row 1 from row 2, using that $v$ and $w$ commute. The resulting last row is identically $0$ because $\alpha_1v\neq 0$ and $r_n$ is at most $2$ by Lemma~\ref{lem:Fox}. 
One obtains the last matrix after substracting the first column from the second column. Finally, $1-w-\alpha_1\alpha_2^{-1}\in\mathbb{Z}\bar\Gamma_n''$ is a unit, as it is non-zero in $\mathbb{Z}\bar\Gamma_0''$, so $\delta_n(C'')=0$ for all $n\geq 1$.


Let $C'$ be an irreducible curve of degree $m'$ such that $C'$ and $C''$ intersect in $6m'$ distinct points, 
and let $C=C'\cup C''$. Using the statement above Lemma~\ref{lem:Fox} and the proofs of Lemmas~\ref{thmTrivialAP} 
and~\ref{thmTrans}, we get that $\delta_0(C)=\left(\text{row corank of a presentation}\right.$ 
$\left.\text{matrix of }H_1(U'',u_0'';\mathcal{K}_0'')\right)-1=(3-r_0)-1=1$, and, by Theorem~\ref{thmCharac}, 
$\delta_n(C)=0$ for all $n\geq 1$. By Corollary~\ref{corAP} and Theorem~\ref{thmAAP}, 
$\Delta_C^{\text{multi}}(t,t_1,t_2)=t-1$.
\end{example}

\bibliographystyle{amsplain}

\end{document}